\documentclass[a4paper, oneside, reqno, 12pt]{amsart}
\pdfoutput=1
\usepackage[utf8]{inputenc}
\usepackage{amsmath, amssymb, amsfonts}
\usepackage{fullpage}
\usepackage{textcomp, cmap, comment}
\usepackage{hyperref}
\usepackage{mathtools} 
\usepackage[nobysame, initials]{amsrefs}

\usepackage{comment}
\usepackage{color, xcolor}

\theoremstyle{plain}
\newtheorem{theorem}{Theorem}
\newtheorem{conjecture}[theorem]{Conjecture}

\newtheorem{lemma}[theorem]{Lemma}
\newtheorem{proposition}[theorem]{Proposition}
\newtheorem{corollary}[theorem]{Corollary}
\theoremstyle{definition}
\newtheorem{problem}[theorem]{Problem}

\newtheorem*{remark*}{Remark}

\DeclareMathOperator{\tr}{tr}
\DeclareMathOperator{\rank}{rank}

\usepackage{pgf,tikz, mathrsfs}
\usetikzlibrary{arrows}

\usepackage{hyperref}
\hypersetup{
  colorlinks   = true, %Colours links instead of ugly boxes
  urlcolor     = blue, %Colour for external hyperlinks
  linkcolor    = red, %Colour of internal links
  citecolor   = green %Colour of citations
}

\title[almost-equidistant sets]{{On almost-equidistant sets}}
\author[A.~Polyanskii]{{A.~Polyanskii}}%
\address{Alexandr Polyanskii
\newline\hphantom{iii} \href{https://mipt.ru/english/}{Moscow Institute of Physics and Technology}
\newline\hphantom{iii} Institutskiy per. 9
\newline\hphantom{iii} Dolgoprudny, Russia 141700
\newline\hphantom{iii} \href{http://iitp.ru/en/about}{Institute for Information Transmission Problems RAS}
\newline\hphantom{iii} Bolshoy Karetny per. 19
\newline\hphantom{iii} Moscow, Russia 127994
\newline\hphantom{iii} Adyghe State University
\newline\hphantom{iii} \href{http://cmcagu.ru/}{Caucasus Mathematical Center}
\newline\hphantom{iii} Pervomayskaya str. 208 
\newline\hphantom{iii} Maykop, Russia 385000
}
\email{\href{mailto:alexander.polyanskii@yandex.ru}{alexander.polyanskii@yandex.ru}}
\urladdr{\url{http://polyanskii.com}}

\thanks{\sc A.~Polyanskii, On almost-equidistant sets}


\begin{document}
%\linenumbers
\thispagestyle{empty}

\begin{abstract}
	A finite set of points in $\mathbb R^d$ is called \textit{almost-equidistant} if among any three distinct points in the set, some two  are at unit distance. We prove that an almost-equidistant set in $\mathbb R^d$ has cardinality at most $5d^{13/9}$.
\end{abstract}

\maketitle
\section{Introduction}

A finite set of (unit) vectors in $\mathbb R^d$ is called \textit{almost orthogonal} if among any three vectors of the set there is at least one orthogonal pair. Erd\H{o}s asked~\cite{R90}: What is the largest cardinality $f_{\pi/2}(d)$ of an almost orthogonal set in $\mathbb R^d$? Using an elegant linear algebraic argument, Rosenfeld~\cite{R90} proved that $f_{\pi/2}(d)=2d$. Clearly, we can take two sets of vectors, each containing $d$ orthogonal vectors, and obtain an almost orthogonal set of size $2d$. Other nice proofs of Rosenfeld's theorem were given in~\cite{D11}*{Theorem~3.5} and~\cite{P02}*{Theorem~6}. Moreover, Deaett exhibited a different example of almost orthogonal sets in $\mathbb R^d$ of size $2d$; see~\cite{D11}*{Theorem~4.11}. 

The following definition naturally generalizes the concept of almost orthogonal sets: A finite set $V$ of (unit) vectors in $\mathbb R^d$ is called \textit{almost $\alpha$-angular} if among any three vectors of $V$, two of them form a fixed angle $\alpha$. In particular, an almost orthogonal set is an almost $\pi/2$-angular set. The following variation of Erd\H{o}s's problem has also been considered: What is the largest cardinality $f_{\alpha}(d)$ of an almost $\alpha$-angular set in $\mathbb R^d$, where $\alpha$ is a fixed angle? Bezdek and L{\'a}ngi~\cite{BL00}*{Theorem~1} found that $f_{\alpha}(d)\leq 2d+2$ for any $\pi/2<\alpha<\pi$ and $d\geq 2$. Moreover, this bound is tight for $\alpha=\alpha_d:=2\arcsin \sqrt{(d+1)/(2d)}$. Indeed, one can take $d+1$ unit vectors in $\mathbb R^d$ such that an angle between any two of them is $\alpha_d$; note that the ends of these vectors correspond to vertices of a regular $d$-simplex inscribed in the unit sphere with center at the origin. Therefore, the union of two such sets of vectors forms an almost $\alpha_d$-angular set of $2d+2$ vectors in $\mathbb R^d$. It is worth pointing out that the argument of Bezdek and L{\'a}ngi is a natural modification of Rosenfeld's approach. The key idea of their proof is based on two facts: The first fact is that the matrix
\[
\mathbf V_{\alpha}:=\langle\mathbf v_i, \mathbf v_j\rangle - \cos \alpha,\text{ where } \{\mathbf v_1, \dots, \mathbf v_n\} \text{ is an almost $\alpha$-angular set},
\]
is positive semidefinite for $\pi/2< \alpha< \pi$; the second fact is that the trace of the matrix $(\mathbf V_{\alpha} - (1-\cos \alpha)\mathbf I_n)^3$ is equal to $0$, where $\mathbf I_n$ is the identity matrix of size $n$. For more details we refer interested readers to \hyperref[subsection:case1]{Subsection~\ref*{subsection:case1}}. When $0<\alpha<\pi/2$, we have~$\cos \alpha>0$, and thus the matrix $\mathbf V_\alpha$ can have one negative eigenvalue, and hence the argument of Bezdek and L{\'a}ngi does not work for small $\alpha$. Unfortunately, we still do not know whether $f_{\alpha}(d)=O(d)$ holds for any $0<\alpha<\pi/2$.

A finite set of $n$ points in $\mathbb R^d$ is called \textit{almost-equidistant} if among any three points in the set, some two are at unit distance. The concept of an almost-equidistant set generalizes the notion of an almost $\alpha$-angular set. Indeed, if a set of unit vectors $\{\mathbf v_1, \dots,\mathbf v_n\}$ is almost $\alpha$-angular, then the set $\{k\mathbf v_1, \dots,k\mathbf v_n\}$ is almost-equidistant, where $k=1/(2\sin (\alpha/2))$. The following question was posed by Zsolt L{\'a}ngi: What is the largest cardinality $f_{\mathrm{ae}}(d)$ of an almost-equidistant set in $\mathbb R^d$? Denote by $G=(V, E)$ a graph such that its vertices are points of an almost-equidistant set $V$ in $\mathbb R^d$ and edges are pairs of vertices that are at unit distance apart. Note that $G$ does not contain a clique of size $d+2$ because it is impossible to embed a regular unit $(d+1)$-simplex in $\mathbb R^d$. Furthermore, it does not contain an anticlique of size $3$ because its vertex set is an almost-equidistant set. By the upper bound on the Ramsey number $R(3, d+2)$ from~\cite{G83}, we have $f_{\mathrm{ae}}(d)\leq 2.4(d+2)^2/\log (d+2)$. Bezdek, Nadz{\'o}di and Visy~\cite{BNV03}*{Statement~(2) in~Theorem~4} proved that $f_{\mathrm{ae}}(2)=7$; notice that the famous Moser spindle has $7$ vertices forming an almost-equidistant set in $\mathbb R^2$. Recently, Balko, P{\'o}r, Scheucher, Swanepoel and Valtr~\cite{BPSSV16} showed that $f_{\mathrm{ae}}(d)=O(d^{3/2})$. Also, they constructed an example of an almost-equidistant set in $\mathbb R^d$ with $2d+4$ vertices for $d\geq 3$ and proved some upper and lower bounds on the largest cardinality of an almost-equidistant set in $\mathbb R^d$ for $d\leq 9$. The main result of the present article is the following improvement.
\begin{theorem}\label{theorem:almost}
$f_{\mathrm{ae}}(d)\leq 5d^{13/9}$.
\end{theorem}
It is worth noting that more recently than this paper was submitted, Kupavskii, Mustafa and Swanepoel~\cite{KMS17} proved $f_{\mathrm{ae}}(d)=O(d^{4/3})$. 

This article is organized in the following way. \hyperref[section:preliminaries]{Section~\ref*{section:preliminaries}} presents some preliminaries. In particular, we discuss some properties of Euclidean distance matrices and unit distance graphs formed by vertices of almost-equidistant sets. In \hyperref[section:proof]{Section \ref*{section:proof}} we prove \hyperref[theorem:almost]{Theorem~\ref*{theorem:almost}}. In \hyperref[section:discussion]{Section~\ref*{section:discussion}} we compare our approach with that of other papers in which some upper bounds on the largest size of an almost-equidistant set are proven. In \hyperref[section:problems]{Section~\ref*{section:problems}} we look at some open problems related to almost-equidistant sets. Along the way, we obtain some facts (\hyperref[lemma:matrixV]{Lemma~\ref*{lemma:matrixV}} and \hyperref[lemma:simplex]{Lemma~\ref*{lemma:simplex}}) that are useful in studying distance graphs in $\mathbb R^d$, for example, unit distance graphs and diameter graphs.
\section{Preliminaries}
\label{section:preliminaries}
\subsection{Properties of Euclidean distance matrices}
Suppose that $\{\mathbf v_1, \dots, \mathbf v_n\}\subset \mathbb R^d$ is an arbitrary set of distinct points, where $n\geq 2$. Let us consider the matrix 
\[
\mathbf V:=\|\mathbf v_i-\mathbf v_j\|^2.
\]
\begin{lemma} \label{lemma:matrixV}
The matrix $\mathbf V$ has exactly one positive eigenvalue.
\end{lemma}
\begin{proof}
Note that $\mathbf V$ is not the zero matrix. As it is symmetric, it has a real eigenvalue different from $0$. Since $\tr(\mathbf V)=0$, the matrix $\mathbf V$ has at least one positive eigenvalue. Let us prove that it has at most one positive eigenvalue. Clearly,
\[
\label{equation:e1}\mathbf V=v^t r+r^t v-2\langle \mathbf v_i, \mathbf v_j\rangle,
\]
where $v:=(\|\mathbf v_1\|^2, \dots, \|\mathbf v_n\|^2)$ and $r:=(1,\dots, 1)$ are row vectors of size $n$, and $\langle \mathbf v_i,\mathbf v_j\rangle$ is the Gram matrix of the vectors $\mathbf v_1,\dots, \mathbf v_n$. 

It is easy to check that the eigenvalues of $v^t r +r^t v$ satisfy the equality
\begin{gather*}
\det (v^t r+r^t v-\mu \mathbf I_n)\\
=(-1)^n\mu^{n-2}\Bigl(\mu^2-\Bigl(2\sum_{i=1}^n \|\mathbf v_i\|^2\Bigr)\mu -\sum_{1\leq i<j\leq n} \left(\|\mathbf v_i\|^2-\|\mathbf v_j\|^2\right)^2\Bigr)=0.
\end{gather*}
Thus the matrix $v^tr+r^tv$ has one positive eigenvalue, one negative eigenvalue and $n-2$ eigenvalues equal to $0$.
Note that the Gram matrix $\langle\mathbf v_i, \mathbf v_j\rangle$ 
is positive semidefinite. In order to finish the proof of \hyperref[lemma:matrixV]{Lemma~\ref*{lemma:matrixV}}, we need to apply \hyperref[theorem:sumofHermitian]{Weyl's Inequality}~\cite{W1912}*{Theorem~1} (or~\cite{Pr94}*{Theorem~34.2.1}).
\begin{theorem}[Weyl's~Inequality]
\label{theorem:sumofHermitian}
Let $\mathbf A$ and $\mathbf B$ be Hermitian matrices of size $n$. Suppose that $\alpha_1\leq \dots\leq \alpha_n$ are eigenvalues of $\mathbf A$, $\beta_1\leq \dots \leq \beta_n$ are eigenvalues of $\mathbf B$, $\gamma_1\leq \dots \leq \gamma_n$ are eigenvalues of $\mathbf A+\mathbf B$. Then $$\gamma_{i}\geq \alpha_j+\beta_{i-j+1} \text{ for } i\geq j \text{ and } \gamma_{i}\leq \alpha_{j}+\beta_{i-j+n} \text{ for } i\leq j.$$
\end{theorem}
\begin{remark*}
Actually, we will use only the inequality $\gamma_{n-1}\leq \alpha_{n-1}+\beta_n$.
\end{remark*}
By \hyperref[theorem:sumofHermitian]{Weyl's Inequality}, the second-largest eigenvalue of $\mathbf V=(x^tr+r^tx)-2\langle\mathbf v_i, \mathbf v_j\rangle$ is not positive, and so $\mathbf V$ has at most one positive eigenvalue. %\hyperref[lemma:matrixV]{Lemma~\ref*{lemma:matrixV}} is proved.
\end{proof}
\begin{corollary}\label{corollary:U}
Let $\mathbf U:=\mathbf V-\mathbf J_n+\mathbf I_n$, where $\mathbf J_n$ is the matrix of ones of size $n\times n$. The matrix~$\mathbf U$ has at most one eigenvalue greater than $1$ and at least $n-d-2$ eigenvalues equal to $1$. 
\end{corollary}
\begin{proof}
It is enough to prove that the matrix $\mathbf W:=\mathbf V-\mathbf J_n$ has at most one positive eigenvalue and $\rank (\mathbf W)\leq d+2$. Note that eigenvalues of $-\mathbf J_n$ are $-n, 0,\dots, 0$. By \hyperref[theorem:sumofHermitian]{Weyl's Inequality}, the second-largest eigenvalue of $\mathbf W$ is not positive, and hence $\mathbf W$ has at most one positive eigenvalue.

Let us prove that $\rank(\mathbf W)\leq d+2$. Obviously,
\[
\mathbf W=v^tr+r^t v-2\langle \mathbf v_i, \mathbf v_j\rangle-r^tr=v^tr+r^t(v-r)-2\langle \mathbf v_i,\mathbf v_j\rangle.
\]
%Since $\rank (\mathbf X+\mathbf Y) \leq \rank (\mathbf X)+ \rank (\mathbf Y)$ for matrices $\mathbf X$ and $\mathbf Y$ of the same size, 
%and $\rank (\mathbf Z\mathbf T)\leq \min(\rank (\mathbf Z), \rank (\mathbf T))$ for matrices $\mathbf Z$ and $\mathbf T$ such that the number of columns in $\mathbf Z$ is equal to the number of rows in $\mathbf T$
Therefore, we have 
\begin{gather*}
\rank(\mathbf W)=\rank\left(v^tr+r^t(v-r)-2\langle \mathbf v_i, \mathbf v_j\rangle\right)
\\ 
\leq\rank (v^t r)+\rank(r^t(v-r))+\rank (-2\langle \mathbf v_i, \mathbf v_j\rangle)\leq 1+1+d=d+2.
\end{gather*}
The last inequality holds because %$\rank (x^t)=\rank (j)=\rank (j^t)=\rank (x-j)=1$ and
$\rank (\langle\mathbf v_i,\mathbf v_j\rangle)\leq d$ for the Gram matrix of vectors $\mathbf v_1,\dots, \mathbf v_n$ in $\mathbb R^d$.
\end{proof}
\subsection{Properties of almost-equidistant sets}
Here and subsequently, we assume that $V$ is an almost-equidistant set in $\mathbb R^d$ and $\mathbf U$ is the corresponding matrix for $V$; see~\hyperref[corollary:U]{Corollary~\ref*{corollary:U}}. Let us prove the following useful lemma about $\mathbf U$.
\begin{lemma}\label{lemma:trace} $\tr(\mathbf U)=\tr(\mathbf U^3)=0$.
\end{lemma}
\begin{proof}
Notice that $\mathbf U$ with $ij$-entry $u_{ij}$ satisfies the following properties:
\begin{enumerate}
\item \label{property1}$u_{ii}=0$ for all $i=1,\dots, n$;
\item \label{property2}$u_{ij}u_{jk}u_{ki}=0$ for all triples $1\leq i,j,k\leq n$.
\end{enumerate}
By \hyperref[property1]{property~(\ref*{property1})}, we have $\tr(\mathbf U)=0$. Using \hyperref[property2]{property~(\ref*{property2})}, we get $\tr (\mathbf U^3)=0$.
\end{proof}
From now on, $G=(V, E)$ stands for the \textit{unit distance graph} formed by the points in~$V$; edges of $G$ are pairs of vertices that are at unit distance apart. We need the following simple lemma.
\begin{lemma}\label{lemma:d+2}
A vertex in $G$ has at most $d+1$ non-neighbors.
\end{lemma}
\begin{proof}
Assume to the contrary that there are $d+2$ vertices in $G$ that are not adjacent to some vertex. Note that these vertices form a regular unit $(d+1)$-simplex because $V$ is an almost-equidistant set, but it is impossible to embed a regular unit $(d+1)$-simplex in $\mathbb R^d$, so this is a contradiction.
\end{proof}
One of the key ingredients of the proof of \hyperref[theorem:almost]{Theorem~\ref*{theorem:almost}} is the following lemma.
\begin{lemma}\label{lemma:simplex}
Let $\mathbf w_i$ for $0\leq i\leq k$ be vertices in $G$. Assume that 
\[
s:=\sum_{i=1}^k \left(\|\mathbf w_0-\mathbf w_i\|^2-1\right)\text{ and } |s| \geq \sqrt{k}.
\]
Then the number of vertices adjacent to all $\mathbf w_i$ for $0\leq i\leq k$ is at most $2d+2$.
\end{lemma}
\begin{proof}
	If, among $\mathbf w_i$, $1\leq i \leq k$, there are vertices adjacent to $\mathbf w_0$, then we may delete them and show that the number of vertices adjacent to $\mathbf w_0$ and its non-neighbors among $\mathbf w_i$, $1\leq i\leq k$, is at most $2d + 2$. Hence, we may assume without loss of generality that $\mathbf w_i$, $1\leq i \leq k$, are non-neighbors of $\mathbf w_0$.
		
%Without loss of generality, assume that $\mathbf w_i$, $1\leq i\leq k$, are non-neighbors of $\mathbf w_0$. Indeed, if among $\mathbf w_i$, $1\leq i\leq k$, there are vertices adjacent to $\mathbf w_0$, then we delete them and show that the number of vertices adjacent to $\mathbf w_0$ and its non-neighbors among $\mathbf w_i$, $1\leq i\leq k$, is at most $2d+2$. 

 Our proof of  is based on the following theorem (see~\cite{DM94}*{Theorem~1}):
\begin{theorem}\label{theorem:centermass}
 Let $X = \{\mathbf x_1,\dots, \mathbf x_n\}$ and $Y = \{\mathbf y_1, \dots, \mathbf y_n\}$ be two point-sets in $\mathbb R^d$. Then 
\[
\sum_{1\leq i,j\leq n} \|\mathbf x_i-\mathbf y_j\|^2 = \sum_{1\leq i<j\leq n} \|\mathbf x_i - \mathbf x_j\|^2 + \sum_{1\leq i< j\leq n}\|\mathbf y_i- \mathbf y_j\|^2 + n^2\|\mathbf x -\mathbf y\|^2,
\]
where $\mathbf x$ and $\mathbf y$ are the barycenters of $X$ and $Y$, respectively, that is, 
\[
\mathbf x =(\mathbf x_1 + \cdots + \mathbf x_n)/n,\ \mathbf y = (\mathbf y_1 + \cdots +\mathbf y_n)/n. 
\]
\end{theorem}
Note that $\mathbf w_1, \dots, \mathbf w_k$ form a regular unit $(k-1)$-simplex because $V$ is an almost-equidistant set and $\|\mathbf w_0-\mathbf w_i\|\ne 1$ for $1\leq i\leq k$. Let $\mathbf o$ be the center of this simplex. We write $S(\mathbf v, r_0)$ for the sphere of radius $r_0$ with center $\mathbf v$. We claim that $S:=S(\mathbf w_1, 1) \cap \dots\cap S(\mathbf w_k, 1)\subset S(\mathbf o, r)$, where $r:= \sqrt{(k+1)/(2k)}$. Indeed, let  $\mathbf o'$ be the orthogonal projection of some point $\mathbf u\in S$ onto the affine hull of $\mathbf w_1,\dots, \mathbf w_k$. Since the triangles $\mathbf u \mathbf o'\mathbf w_i$, $1\leq i\leq k$, are right triangles with the common cathetus $\mathbf o'\mathbf u$ and the congruent hypotenuses $\mathbf u\mathbf w_i$, they are congruent, and consequently $\mathbf o'$ is the center of $\mathbf w_1, \dots, \mathbf w_k$, i.e., $\mathbf o'=\mathbf o$. Using the Pythagorean Theorem for the triangle $\mathbf u\mathbf o\mathbf w_1$ and the fact that the circumscribed radius of a unit $(k-1)$-simplex is $\sqrt{(k-1)/(2k)}$, we have $\|\mathbf u-\mathbf o\|=r$. 

Applying \hyperref[theorem:centermass]{Theorem~\ref*{theorem:centermass}} for $X=\{\mathbf w_0,\dots,\mathbf w_0\}$ and $Y=\{\mathbf w_1,\dots, \mathbf w_k\}$, we obtain
\[
k(s+k)=\frac{k(k-1)}{2}+k^2x^2, \text{ where $x:=\|\mathbf w_0-\mathbf o\|$.}
\]
Since $|s|\geq \sqrt k$, we have
\begin{equation}
\label{equation:+}
x^2= \frac{s}{k}+\frac{k+1}{2k}\geq\frac{1}{\sqrt{k}}+\frac{k+1}{2k}=\left(\frac{1}{\sqrt{2}}+\frac{1}{\sqrt{2k}}\right)^2
\end{equation}
or
\begin{equation}
\label{equation:-}
x^2= \frac{s}{k}+\frac{k+1}{2k}\leq-\frac{1}{\sqrt{k}}+\frac{k+1}{2k}=\left(\frac{1}{\sqrt{2}}-\frac{1}{\sqrt{2k}}\right)^2.
\end{equation}

Note that vertices in $G$ adjacent to all $\mathbf w_i$, $0\leq i\leq k$, lie on the sphere $w:=S(\mathbf w_0,1)\cap S(\mathbf o, r)$. Let us prove that its radius~$r'$ is at most $1/\sqrt{2}$. Suppose $\mathbf z\in w$. Then $\|\mathbf z - \mathbf w_0\|=1$ and $\|\mathbf z - \mathbf o \| = r$. Note that the center of $w$ must coincide with the orthogonal projection $\mathbf z'$ of $\mathbf z$ onto the line passing through $\mathbf o$ and $\mathbf w_0$, and hence $r'=\|\mathbf z-\mathbf z'\|$. Denote by $\theta$ the angle $\angle \mathbf z \mathbf w_0\mathbf o$. Using the Law of Cosines, we have
\begin{gather*}
\cos \theta=\frac{\|\mathbf z-\mathbf w_0\|^2+\|\mathbf w_0-\mathbf o\|^2-\|\mathbf z-\mathbf o\|^2}{2\|\mathbf z-\mathbf w_0\|\|\mathbf w_0-\mathbf o\|}\\
=\frac{1+x^2-(k+1)/2k}{2x}=\frac{k-1}{4kx}+\frac{x}{2}=:g(x).
\end{gather*}
It is easily seen that $\sqrt{(k-1)/(2k)}$ is the only point of local minimum of $g(t)$ for $t>0$, and
\[
\frac{1}{\sqrt{2}}-\frac{1}{\sqrt{2k}}<\sqrt{\frac{k-1}{2k}}<\frac{1}{\sqrt{2}}+\frac{1}{\sqrt{2k}}.
\]
Therefore, because \eqref{equation:+} and \eqref{equation:-} hold, we have 
\begin{gather*}
	g(x)\geq \min\left\{g\left(\frac{1}{\sqrt{2}}-\frac{1}{\sqrt{2k}}\right),
	g\left(\frac{1}{\sqrt{2}}+\frac{1}{\sqrt{2k}}\right)\right\}=\frac{1}{\sqrt{2}}.
\end{gather*}
Thus $\cos\theta\geq 1/\sqrt{2}$, and so $r'=\|\mathbf z-\mathbf w_0\|\sin \theta \leq 1/\sqrt{2}$.

Suppose that  $V\cap w=\{\mathbf u_1, \dots, \mathbf u_m\}$. Assume that $\|\mathbf u_i-\mathbf u_j\|=1$ for some $1\leq i,j \leq m$. Since the radius of $w$ is at most $1/\sqrt{2}$, we have $\angle \mathbf u_i\mathbf z'\mathbf u_j=\alpha\geq \pi/2$, where $\alpha$ is a fixed angle. Indeed, by the Law of Cosines, we obtain that
\[
\cos \angle \mathbf u_i\mathbf z'\mathbf u_j=\frac{\|\mathbf u_i-\mathbf z'\|^2+\|\mathbf u_j-\mathbf z'\|^2-\|\mathbf u_i-\mathbf u_j\|^2}{2\|\mathbf u_i-\mathbf z'\|\|\mathbf u_j-\mathbf z'\|}=\frac{2(r')^2-1}{2(r')^2}\leq 0
\]
is a fixed number. Therefore, $\{\mathbf u_1-\mathbf z',\dots,\mathbf u_m-\mathbf z'\}$ is an almost $\alpha$-angular set in $\mathbb R^d$ (note that these are not unit vectors). By the results of Rosenfeld~\cite{R90} and Bezdek--L\'angi~\cite{BL00}, we get $m\leq 2d+2$.
\end{proof}

\section{Proof of Theorem~\ref{theorem:almost}}
\label{section:proof}
Since $f_{\mathrm{ae}}(2)=7$ (see~\cite{BNV03}), we can assume that $d\geq 3$. By \hyperref[corollary:U]{Corollary~\ref{corollary:U}}, we have two cases: The matrix $\mathbf U$ does not have eigenvalues greater than $1$ or has exactly one such eigenvalue. The proof of the first case is an almost word-for-word
repetition of the proof of Theorem 1 in~\cite{BL00}. The proof of the second case involves new ideas, in particular the \hyperref[theorem:gershgorin]{Gershgorin Circle Theorem}~\cite{G31} and \hyperref[lemma:simplex]{Lemma~\ref*{lemma:simplex}}.
\subsection{Proof of the first case}\label{subsection:case1} Assume that the matrix $\mathbf U$ does not have eigenvalues greater than $1$. Denote by $\lambda_1, \dots, \lambda_k$ the eigenvalues of $\mathbf U$ that are less than $1$. By \hyperref[corollary:U]{Corollary~\ref*{corollary:U}} and \hyperref[lemma:trace]{Lemma~\ref{lemma:trace}}, we have $k\leq d+2$ and
\begin{gather}
\label{equation:sum}
\sum_{i=1}^k(-\lambda_i)=\sum_{i=1}^k (-\lambda_i)^3=n-k.
\end{gather}
In order to finish the proof of the current case, we need the following lemma (see~\cite{BL00}*{Lemma~1}).
\begin{lemma}
	\label{lemma:bezdeklangi}
	Let $x_1,\dots, x_m$ be real numbers with the property that there exists $y>0$ such that $x_i\geq -y$ for $i=1,\dots,m$ and $\sum_{i=1}^mx_i=(m+l)y$, where $l\geq 0$. Then
	\[
	\sum_{i=1}^m x_i^3\geq (m+3l)y^3.
	\] 	
\end{lemma}
Assume that $n>2k$. Introducing the notation $l=n-2k>0$, we can rewrite~\eqref{equation:sum}~as
\[
	\sum_{i=1}^k(-\lambda_i)=\sum_{i=1}^k(-\lambda_i)^3=(k+l).
\]
Thus \hyperref[lemma:bezdeklangi]{Lemma~\ref*{lemma:bezdeklangi}} for $y=1$ implies that
\[
	\sum_{i=1}^k(-\lambda_i)^3\geq(k+3l),
\]
and so this is a contradiction that completes the proof of the first case.

\subsection{Proof of the second case} Assume that the matrix $\mathbf U$ has exactly one eigenvalue $\lambda>1$. Denote by $\lambda_1, \dots, \lambda_k$ the eigenvalues of $\mathbf U$ that are less than $1$. By \hyperref[corollary:U]{Corollary~\ref*{corollary:U}} and \hyperref[lemma:trace]{Lemma~\ref{lemma:trace}}, we have $k\leq d+1$ and
\[
\lambda+\sum_{i=1}^{k}\lambda_i+(n-k-1)=
\lambda^3+\sum_{i=1}^{k}\lambda_i^3+(n-k-1)=0.
\]
Without loss of generality, assume that $\lambda_1,\dots, \lambda_l\leq 0$ and $\lambda_{l+1},\dots, \lambda_k>0$, where $1\leq l\leq k$. Therefore, we have
\[
\sum_{i=1}^{l}(-\lambda_i)=\lambda+(n-k-1)+\sum_{i=l+1}^{k}\lambda_i\geq n+\lambda-d-2\geq n-d-1,
\]
so
\begin{gather*}
\lambda^3=\sum_{i=1}^l(-\lambda_i)^3+\sum_{i=l+1}^k(-\lambda_i)^3-(n-k-1)\geq \frac{(-\lambda_1-\dots-\lambda_l)^3}{l^2}-(k-l)-(n-k-1)
\\
\geq\frac{(n-d-1)^3}{(d+1)^2}-n.
\end{gather*}
Suppose, contrary to our claim, that $n= 5d^{\beta}\geq 4 d^{\beta}+d+1$ for some $\beta> 13/9$. Thus
\[
\lambda^3\geq \frac{ 64d^{3\beta}}{2d^2}-5d^{\beta}\geq 27 d^{3\beta-2}, \text{ and hence } \lambda\geq 3d^{\beta-2/3}.
\]

We need the \hyperref[theorem:gershgorin]{Gershgorin Circle Theorem}~\cite{G31} (or~\cite{Pr94}*{Problem~34.1}).
\begin{theorem}[Gershgorin Circle Theorem]
	\label{theorem:gershgorin}
	Every eigenvalue of a matrix $(a_{ij})$ over $\mathbb C$ of size $n\times n$ belongs to one of the disks 
    \[
    \Bigl\{z\in\mathbb C:|a_{kk} - z| \leq 
    \sum_{1\leq j\leq n, j\ne k} |a_{kj}|\Bigr\} \text{ for } k=1,\dots, n.
    \]
\end{theorem}
Using the \hyperref[theorem:gershgorin]{Gershgorin Circle Theorem} for the matrix $\mathbf U$, we can assume
\[
\sum_{j=2}^n\left| \|\mathbf v_1-\mathbf v_j\|^2-1\right|\geq \lambda\geq 3d^{\beta-2/3}.
\]
By \hyperref[lemma:d+2]{Lemma~\ref*{lemma:d+2}}, on the left-hand side there are at least $n-d-2$ terms equal to~$0$. Thus, there is no loss of generality in assuming
\[
\sum_{j=2}^{d+2}\left|\|\mathbf v_1-\mathbf v_j\|^2-1\right|\geq 3d^{\beta-2/3}.
\]
With the notation $t:=\lceil d^{4/9}\rceil$, we can assume
\[
\sum_{j=2}^{2t+1}\left|\left(\|\mathbf v_1-\mathbf v_j\|^2-1\right)\right|\geq 3 d^{\beta-2/3}\cdot\frac{2 d^{4/9}}{d+1}\geq 4d^{2/9}> 3\sqrt{t}.
\]
The last two inequalities hold since we are assuming $d\geq 3$. Clearly, there is a subset $J\subseteq \{2,\dots, 2t+1\}$ such that $|J|\leq t$, $\|\mathbf v_1-\mathbf v_j\|^2-1$ are of the same sign for $j\in J$ and
\[
\Bigl|\sum_{j\in J}\left(\|\mathbf v_1-\mathbf v_j\|^2-1\right)\Bigr|>\sqrt{t}.
\]

By \hyperref[lemma:simplex]{Lemma~\ref*{lemma:simplex}}, the number of vertices in $G$ that are adjacent to all $\mathbf v_j$ for $j\in \{1\}\cup J$ is at most $2d+2$. By \hyperref[lemma:d+2]{Lemma~\ref*{lemma:d+2}}, the number of vertices that are not adjacent to at least one $\mathbf v_j$ with $j\in\{1\}\cup J$ is at most $(t+1)(d+1)$. Therefore, $n\leq (2d+2)+(t+1)(d+1)< 5d^{13/9}$, so this is a contradiction that finishes the proof of the second case.
\qed

\begin{corollary}\label{corollary:alpha} We have
	$f_{\alpha}(d)\leq 5d^{13/9}$ for $0<\alpha<\pi/2$.
\end{corollary}
\begin{proof}
Suppose that $\{\mathbf v_1, \dots, \mathbf v_n\}$ is an almost $\alpha$-angular set of unit vectors in $\mathbb R^d$. Clearly, the set $\{k\mathbf v_1, \dots, k\mathbf v_n\}$ is an almost-equidistant set in $\mathbb R^d$, where $k=1/(2\sin (\alpha/2))$. Therefore, $n\leq 5d^{13/9}$.
\end{proof}
\section{Discussion}\label{section:discussion}
Now we compare our proof with other proofs (see~\cite{BPSSV16} and~\cite{KMS17}) of upper bounds on the largest size of an almost-equidistant set.

The common idea is to estimate the sizes of two subsets $T_1$ and $T_2$ of an almost-equidistant set that are constructed using some third subset $T$. In the present paper, $T$ is the union of a clique and a vertex that is not adjacent to the clique. Note that we use Rosenfeld's method~\cite{R90} and the \hyperref[theorem:gershgorin]{Gershgorin Circle Theorem} to find $T$. But the choice of $T$ in other papers is quite simple: In~\cite{BPSSV16} the subset $T$ is any clique of size $\lfloor \sqrt{d}\rfloor$, and in~\cite{KMS17} the subset $T$ is a clique of maximum cardinality. 

In all papers, the first subset $T_1$ contains vertices adjacent to all points of $T$ (as in the present paper and in~\cite{BPSSV16}) or to almost all points of $T$ (as in~\cite{KMS17}). The second subset $T_2$ is just the complement of $T_1$. To bound the number of points in $T_2$, we apply the trivial bound as in~\cite{BPSSV16}, but the authors of~\cite{KMS17} use double counting. To estimate the size of $T_1$, the authors of~\cite{BPSSV16} and~\cite{KMS17} apply a lemma\footnote[1]{This lemma claims that $\rank A\geq (\tr A)^2/\tr A^2$ for a Hermitian matrix $A$.} that Deaett~\cite{D11}*{Lemma~3.4} used to prove $f_{\pi/2}(d)=2d$. We bound the size of $T_1$ using a new tool (\hyperref[lemma:simplex]{Lemma~\ref*{lemma:simplex}}) that is based on the results of Rosenfeld and Bezdek--L\'angi.

So the key difference between our approaches is that we try to follow Rosenfeld's proof of the fact that an almost orthogonal set in $\mathbb R^d$ has at most $2d$ points, but the authors of other articles followed Deaett's proof. 
\section{Open problems}\label{section:problems}

Unfortunately, we were not able to prove the following natural conjecture. 
\begin{conjecture}
	\label{conjecture:almost}
	$f_{\mathrm{ae}}(d)=O(d)$.
\end{conjecture}
But we can prove it for almost-equidistant sets of large diameter.
\begin{proposition}
\label{proposition:diameter}
Suppose that the diameter of an almost-equidistant set $V$ in $\mathbb R^d$ is at least $\sqrt{2}$, then $|V|\leq 4d+4$.
\end{proposition}
\begin{proof}
Assume that $\|\mathbf v-\mathbf u\|\geq \sqrt{2}$ for $\mathbf v, \mathbf u\in V$. By \hyperref[lemma:simplex]{Lemma~\ref*{lemma:simplex}}, the number of points in $V$ that are at unit distance from $\mathbf v$ and $\mathbf u$ is at most $2d+2$. By \hyperref[lemma:d+2]{Lemma~\ref*{lemma:d+2}}, the number of points in $V$ that are not at unit distance from $\mathbf u$ or $\mathbf v$ is at most $2d+2$. Thus $|V|\leq 4d+4$.
\end{proof}
Clearly, \hyperref[proposition:diameter]{Proposition~\ref{proposition:diameter}} implies that, in order to prove \hyperref[conjecture:almost]{Conjecture~\ref{conjecture:almost}}, we can assume that the diameter of an almost-equidistant set does not exceed $\sqrt{2}$.
Therefore, it would be natural to ask the following question.
\begin{problem}
	\label{problem:diameter}
	A subset of $\mathbb R^d$ is called an \textit{almost-equidistant diameter set} if it is an almost-equidistant set in $\mathbb R^d$ and has diameter $1$. What is the largest cardinality of an almost-equidistant diameter set in $\mathbb R^d$?
\end{problem}
\section*{Acknowledgments}
We are grateful to \href{http://math.bme.hu/~zlangi/}{Zsolt L\'angi} for posing his question about the largest cardinality of almost-equidistant sets, and to \href{http://personal.lse.ac.uk/swanepoe/index.html}{Konrad Swanepoel} for several interesting and stimulating discussions on the problem. We thank an anonymous referee for valuable comments that helped to significantly improve the presentation of the paper and the clarity of the proofs.

The author is supported in part by ISF grant no.\ 409/16, and by the Russian Foundation for Basic Research through grant nos.\ 15-01-99563 A, 15-01-03530 A, and by the Leading Scientific Schools of Russia through grant no. NSh-6760.2018.1.

The work was done when the author was a postdoctoral fellow at the Technion.
%\bigskip

\bibliographystyle{apalike}
\bibliography{biblio}
\end{document}